\documentclass[10pt]{amsart}
\usepackage[english]{babel}

\usepackage{amscd}
\usepackage{amsmath}
\usepackage{amsfonts}
\usepackage{mathrsfs}
\usepackage{bm}
\usepackage{enumerate}
\usepackage{xcolor}

\usepackage[margin=0.75in]{geometry}
\usepackage{amsthm}
\usepackage{amssymb}
\usepackage{tabto}
\usepackage{nccmath}
\usepackage[colorlinks, citecolor=blue, linkcolor=red, pdfstartview=FitB]{hyperref}
\usepackage{esvect}
\usepackage{graphicx}
\usepackage{float}
\usepackage{color}

\usepackage{cite}
 
\theoremstyle{definition}

\newtheorem{theorem}{Theorem}[section]

\newtheorem{example}{Example}[subsection]
\newtheorem{lemma}[theorem]{Lemma}
\newtheorem{corollary}[theorem]{Corollary}

\def\IC{\mathbb{C}}

\def\IR{\mathbb{R}}
\def\IS{\mathbb{S}}
\def\IT{\mathbb{T}}

\def\B{\mathcal{B}}

\def\H{\mathcal{H}}

\def\K{\mathcal{K}}

\def\O{\mathcal{O}}

\def\R{\mathcal{R}}

\def\R{\mathcal{R}}

\newcommand\norm[1]{\left\| #1 \right\|}

\newcommand{\comment}[1]{}

\renewcommand\eqs[1]{equation \eqref{eq:#1}}

\newcommand{\rank}{\operatorname{rank}}

\renewcommand{\d}{\partial}

\begin{document}
\title{Lifting Sylvester equations: singular value decay for non-normal coefficients}

\author{Rapha\"{e}l Clou\^{a}tre}

\address{Department of Mathematics, University of Manitoba, Winnipeg, Manitoba, Canada R3T 2N2}

\email{raphael.clouatre@umanitoba.ca}

\author{Brock Klippenstein}
\email{klippe11@myumanitoba.ca}

\author{Richard Mika\"{e}l Slevinsky}
\email{richard.slevinsky@umanitoba.ca}

\thanks{R.C. was partially supported by an NSERC Discovery Grant (RGPIN-03600-2022). RMS was partially supported by an NSERC Discovery Grant (RGPIN-2017-05514)}

\maketitle

\begin{abstract}
We aim to find conditions on two Hilbert space operators $A$ and $B$ under which the expression $AX-XB$ having low rank forces the operator $X$  itself to admit a good low rank approximation. It is known that this can be achieved when $A$ and $B$ are normal and have well-separated spectra.  In this paper, we relax this normality condition, using the idea of operator dilations. The basic problem then becomes the lifting of Sylvester equations,  which is reminiscent of the classical commutant lifting theorem and its variations. Our approach also allows us to show that the (factored) alternating direction implicit method for solving Sylvester equaftions $AX-XB=C$ does not require too many iterations, even without requiring $A$ to be normal.
\end{abstract}

\section{Introduction}
When solving linear two-dimensional partial differential equations numerically \cite{townsend2015automatic}, one can start by discretizing the equation to obtain a matrix equation of the form
\begin{equation}\sum_{k=1}^N A_kXB_k=C,\label{eq:generalizedsylvester}\end{equation}
where $X$ is a discretized approximation of the solution and $A_k$,  $B_k$, and $C$ depend on the structure of the differential equation.  The \eqs{generalizedsylvester} is known as a generalized Sylvester equation.  Given a simpler linear two-dimensional partial differential equation such as the heat or wave equation \cite{townsend2015automatic,fortunato2020fast}, the corresponding matrix equation can be expressed as
\begin{equation}AX-XB=C.\label{eq:sylvester}\end{equation}
This is known as a Sylvester equation with coefficients $A$ and $B$, right-hand side $C$, and solution $X$.  
We refer the interested reader to \cite{bhatia1997and} for a survey on Sylvester equations.

Additional applications of Sylvester equations arise in \cite{klippenstein2022fast} where they are used to find the connection coefficients between families of orthogonal polynomials.  Certain classes of matrices such as Cauchy matrices solve very well-structured Sylvester equations with so-called low displacement rank \cite{beckermann2019bounds}.

Our starting point  is a result of Beckermann and Townsend. Before we can state it, we introduce some notation. Given a bounded linear operator $T$ on a Hilbert space, we denote its spectrum by $\sigma(T)$. When $T$ is compact, we denote by $\{s_n(T)\}$ the countable set of its non-zero singular values, arranged in decreasing order. 
Next, for each integer $k\geq 0$, we denote by $\R_k$ the set of rational functions of the form $p/q$, where both $p$ and $q$ are polynomials of degree at most $k$. Then, given two subsets $E$ and $F$ of the complex plane, the corresponding \emph{Zolotarev number} is defined to be
\begin{equation}Z_k(E,F) = \inf_{r\in\R_k}\frac{\sup_{z\in E} |r(z)|}{\inf_{z\in F}|r(z)|}.
\label{eq:zolotarev}\end{equation}
We can now state the result from \cite{beckermann2019bounds} that we aim to extend. Strictly speaking, the original result is only stated for matrices, but the argument therein adapts verbatim to cover general bounded linear operators on possibly infinite-dimensional Hilbert spaces.

\begin{theorem}\label{T:BT}
Consider the Sylvester equation $AX-XB=C$, where $A,B,C,X$ are bounded linear operators on Hilbert spaces. Assume that $A$ and $B$ are normal, that $X$ is compact, and that $C$ has finite rank equal to $v$. Then, \begin{equation} s_{\ell+vk}(X) \leq Z_k\left(\sigma(A),\sigma(B)\right)s_\ell(X)\label{eq:beckermanntownsend}
\end{equation}
for each positive integer $k,\ell$.
\end{theorem}

An important generic property of Zolotarev numbers is that the better separated $E$ and $F$ are, the more rapidly $Z_k(E,F)$ decays to zero as $k\to\infty$. Therefore, Theorem \ref{T:BT} implies that if the spectra of $A$ and $B$ are well separated, then the singular values of $X$ decay quickly. This is advantageous, in light of the classical Eckart-Young Theorem \cite{davidsonnest,  eckart1936approximation,  mirsky1960symmetric}. Indeed, using the formula
$$s_{k+1}(X)=\inf_{\rank(R) \leq k} \norm{X-R}$$
we see that rapid decay of the sequence of singular numbers implies that $X$ has a good low rank approximation.
This property is clearly desirable for computational reasons. Further, there are algorithms which have their computational complexity dependent on the rank of certain matrices; see \cite{townsend2014computing} for algorithms solving partial differential equations numerically which are quicker the lower the rank of the forcing.  Additionally, we will see in this paper how the time complexity of the factored alternating direction implicit method is quadratic in general, but in fact is linear if $C$ has low rank and certain conditions on $A$ and $B$ are satisfied.

The basic question motivating our work is whether the normality condition from Theorem \ref{T:BT} can be relaxed. It is easy to see that any kind of special behaviour of the singular values of $X$ certainly cannot be expected unconditionally. Indeed, let $B,C$ and $X$ be square matrices of the same size, with $X$ invertible. Letting $A=(C+XB)X^{-1}$, it follows that $AX-XB=C$, thereby illustrating that $X$ can have full rank regardless of that of $C$. 

In view of this obstruction, the next natural guess may be that fast decay of the singular values of $X$ (when $C$ has low rank) might depend on the proximity of $X$ to normal matrices. It was shown in  \cite{baker2015fast}  that this is not the case, notwithstanding the many possible definitions of distance to normality. 
Taken together, these observations show that the appropriate direction in which to extend Theorem \ref{T:BT} is not obvious at first glance. 

The alternating direct implicit (ADI) method  is an iterative algorithm used to obtain an approximate solution of a Sylvester equation.  If in \eqs{sylvester}, $A$ and $B$ are normal and $X_k$ is the approximate solution after $k$ iterations,  then we have
$$\norm{X-X_k}\leq Z_k\left(\sigma(A),\sigma(B)\right) \norm{X}.$$
This bound is reminiscent of the conclusion of Theorem \ref{T:BT}, and thus one may wonder if the normality condition can be relaxed in this case as well. We will show that this is indeed the case.

The structure of this paper is as follows. Section \ref{zolotarev} is concerned with Zolotarev numbers. In Theorem \ref{lineandcirclebound}, we derive a crucial bound that appears to be new for the Zolotarev numbers over the unit disk and an interval.  Section \ref{prelim} introduces some required background on operator theory. Section \ref{dilatesylvester} contains our main technical results (Theorems \ref{T:mainfinite} and \ref{T:maininfinite}), which show  that a given Sylvester equation can be ``lifted" to another one with better behaved coefficients while preserving the  information that is relevant for our current purposes. 
Notably, this can be viewed as an extension of the so-called commutant lifting theorem or intertwiner lifting theorem \cite{paulsen2002completely}, which corresponds to the case where $C=0$. The proof of our result requires a different approach, however. We offer two applications of our main results. First, in Corollary \ref{fastdecay} we extend Theorem \ref{T:BT} and establish fast decay of singular values for the solution
of a Sylvester equation. Second, in Section \ref{applications},  we show how a Sylvester equation can be solved quickly without requiring normality of both coefficients.  Finally,  Section \ref{numericalresult} illustrates how our theory can be applied to solve certain partial integro-differential equations.

Certain proofs in this paper are abridged versions of those found in the second author's master's thesis~\cite{klippenstein2022singular}. We refer the interested reader to it for further details on the results presented herein.

\section{Zolotarev Numbers}\label{zolotarev}

In \cite{zolotarev1877application},  Zolotarev introduced four foundational problems in polynomial and rational approximation theory.  We discuss the third problem here,  which concerns finding a rational function minimized over one set while maximized over another.  We begin with an explicit formula for the Zolotarev numbers over two intervals given in \cite{beckermann2019bounds, zolotarev1877application}.  Then,  we use that result to obtain an upper bound on the Zolotarev numbers over the unit circle and an interval.

Useful properties of Zolotarev numbers are that they decrease with respect to degree,  and increase with respect to set inclusion. Further,  if $T$ is a M\"{o}bius transform,  then $Z_k\left(T(E),T(F)\right)=Z_k(E,F)$ since $\R_k$ is closed with respect to M\"obius transformations.  If $r_k^*$ is an infimizer in \eqs{zolotarev}, we call it an extremal function for $Z_k(E,F)$.

It is important to determine when the Zolotarev numbers of two sets decay. We will say that two complex sets $E$ and $F$ are \emph{well separated} if there exist constants $C\geq 0$ and $0\leq \alpha<1$, both possibly depending on $E$ and $F$, such that
$$Z_k(E,F) \leq C\alpha^k$$
for each $k$.  As we shall see, two examples of sets which are well separated are disjoint intervals, as well as the unit circle and an interval, assuming they do not intersect.

For special sets such as disjoint real intervals, Zolotarev numbers have been extensively studied, as can be seen in \cite{gonvcar1969zolotarev, saff2013logarithmic, lebedev1977zolotarev, beckermann2019bounds}. However, the result we are most interested in is the following theorem.

\begin{theorem}[Beckermann and Townsend,  \cite{beckermann2019bounds}]\label{lines}
Let $E=[a,b]$ and $F=[c,d]$ be disjoint real intervals.  Denote
$$ \rho = \exp\left(\frac{\pi K(1/\alpha)}{K(\sqrt{1-\alpha^{-2}})}\right),$$
where $K(\cdot)$ is the complete elliptic integral of the first kind,  \cite[Chapter 16]{NIST:DLMF},  and
$$\alpha=-1+2\gamma+2\sqrt{\gamma^2-\gamma} ,\qquad \gamma = \left|\frac{(c-a)(d-b)}{(d-a)(c-b)}\right| > 1.$$
Then,
\begin{equation}Z_k(E,F) =4\rho^{-2k}\prod_{n=1}^\infty \frac{\left(1+\rho^{-8nk}\right)^4}{\left(1+\rho^{4k}\rho^{-8nk}\right)^4}\leq 4\left[\exp\left(\frac{\pi^2}{2\ln (16\gamma)}\right)\right]^{-2k}.\label{eq:bound}\end{equation}
\end{theorem}

Now we relate the Zolotarev numbers over an interval and a circle to the Zolotarev numbers over two real intervals.

\begin{theorem}\label{lineandcirclebound}
Let $E=[a,b]$ for $1<a<b$ and $F=\{z\in\IC:|z|=1\}$. Then for any $\alpha>(\frac{1+a}{1-a})^2$,
$$Z_{2k}(E,F) \leq Z_k\left(\left[0,\frac{1}{\alpha}\right],\left[\frac{1}{\alpha-d^2},\frac{1}{\alpha-c^2}\right]\right)$$
where
$$c=\frac{1+a}{1-a},\qquad d=\frac{1+b}{1-b}.$$
\end{theorem}
\begin{proof}
Consider the M\"{o}bius transformation given by
$$T_1(z) = \frac{1+z}{1-z}i.$$
For a real number $\phi$, we find
$$T_1(e^{i\phi}) = \frac{\sin \phi}{\cos\phi-1},$$
and thus $T_1$ sends the unit circle to the extended real line with $T_1(1)= \infty$. Additionally,  $T_1$ maps the extended real line to the extended imaginary axis.  Furthermore, if $x>1$, then $T_1(x)$ is on the negative imaginary axis. Thus, our original problem turns into finding the Zolotarev numbers where $G=\IR$ is the set of real numbers, and $H=i[c,d]$ where $c<d<0$.
Put $f(z)=z^2$. Then 
\begin{equation}Z_{2k}(G,H) \leq Z_k(f(G),f(H))=Z_k\left([0,\infty),\left[-c^2, -d^2\right]\right).\label{eq:squarelines}\end{equation}
Finally, given any $\alpha >c^2$ the M\"obius transformation
$$T_2(z) = \frac{1}{z+\alpha},$$
confirms that
$$Z_k\left([0,\infty),\left[-c^2,-d^2\right]\right) =Z_k\left(\left[0,\frac{1}{\alpha}\right],\left[\frac{1}{\alpha-d^2},\frac{1}{\alpha-c^2}\right]\right).$$
\end{proof}

By combining the two previous theorems,  we obtain a relatively simple bound on the Zolotarev numbers over the unit circle and an interval.

\begin{corollary}\label{lineandcircleboundcorollary}
Let $E=[a,b]$ for $1<a<b$ and $F=\{z\in\IC:|z|=1\}$. Then,
$$Z_{2k}(E,F)  \leq 4\left[\exp\left(\frac{\pi^2}{2\ln (16\gamma)}\right)\right]^{-2k} ,\qquad \gamma = \left[\frac{(a+1)(1-b)}{(1-a)(b+1)}\right]^2 > 1.$$
\end{corollary}

It should be noted that the bound given by Theorem \ref{lineandcirclebound} is not optimal.  Although the majority of the proof consists of equalities, the step that involves squaring the orthogonal lines gives rise to an inequality. To illustrate the sharpness of the bound, we set
$$r_1(z) = \frac{z-w}{z-\frac{1}{w}},\qquad w= \frac{a+b}{2}.$$
One can show that this implies
$$Z_1(E,F) \leq \frac{\frac{b}{a}-1}{a+b-\frac{2}{a}}.$$
For $k>1$, we can take $r_k = r_1^k$ to obtain
\begin{equation}Z_k(E,F) \leq \left(\frac{\frac{b}{a}-1}{a+b-\frac{2}{a}}\right)^k.\label{eq:bound2}\end{equation}
Figure \ref{fig:comparisons} compares the bounds given by Theorem \ref{lineandcirclebound} and Corollary \ref{lineandcircleboundcorollary} to the bound given in \eqs{bound2}.

\begin{figure}[H]
\centering
\includegraphics[scale=0.5]{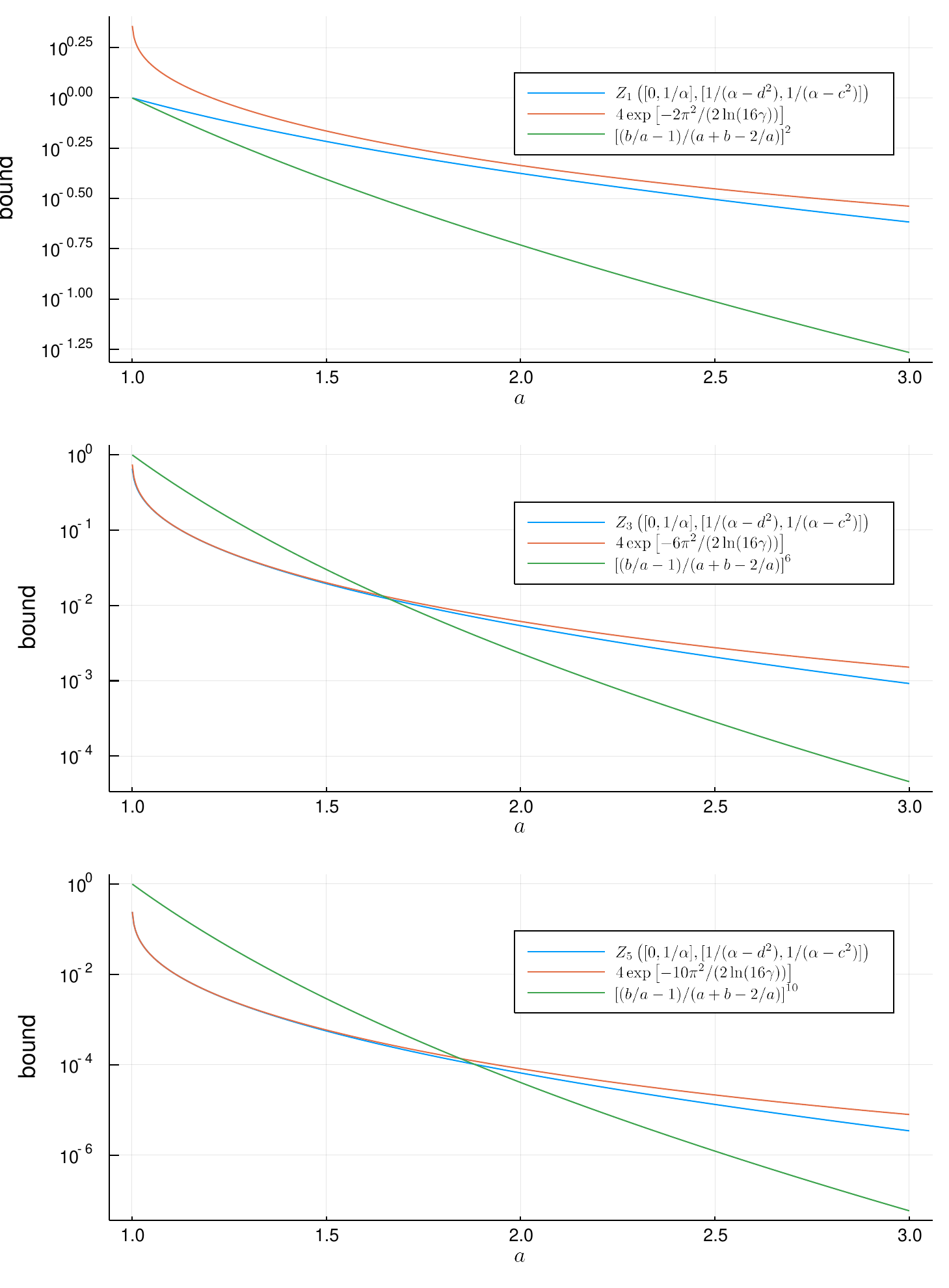}
\caption{These plots compare the bounds on $Z_{2k}(E,F)$ derived in Theorem \ref{lineandcirclebound},  Corollary \ref{lineandcircleboundcorollary},  and \eqs{bound2}. The top, middle and bottom compare the bounds for $k=1,3,5$, respectively. We vary $a$ and fix $b=10$. }
\label{fig:comparisons}
\end{figure}

\section{Operator theoretic prelimineries}\label{prelim}

Throughout this paper, $\H$ will denote a Hilbert space. Given another Hilbert space $\H'$ and a linear operator $T:\H\to\H'$, its norm is defined as
$$\norm{T}=\sup_{0\ne x\in\H}\frac{\norm{Tx}}{\norm{x}}.$$
The space of bounded linear operators from $\H$ to $\H'$ will be denoted by $\B(\H,\H')$; when $\H=\H'$ we simply write $\B(\H)$.
Sometimes we will restrict our attention to finite-dimensional spaces, in which case operators can be interpreted as matrices.  For positive integers $m$ and $n$, we let $\IC^{m\times n}$ denote the space of $m\times n$ complex matrices. 

We record a standard fact for future use.

\begin{lemma}\label{L:spec}
Let $T\in\B(\H)$. Then, $\sigma(T)\subset \{z\in \IC:|z|>1\}$ if and only if the series $\sum_{n=0}^\infty \|T^{-n}\|$ converges.
\end{lemma}
\begin{proof}
Observe that $\sigma(T^{-1})=\{1/\lambda:\lambda\in \sigma(T)\}$. The result then follows from the  spectral radius formula along with the root test for convergence.
\end{proof}

\subsection{Dilations}

The main idea behind our main contributions  is that of a dilation. Let $\K$ be a Hilbert space containing a closed subspace $\H$. Let $P:\K\to \H$ denote the orthogonal projection. Let $B$ be an operator on $\K$ and $A$ be an operator on $\H$. We say that $B$ is a \emph{dilation} of $A$ if
$
P B|_{\H}=A.
$
This can be reformulated as follows. 

Decompose $\K$ as $\H\oplus \H^\perp.$ With respect to this decomposition, $B$ has a $2\times 2$ block matrix decomposition. The fact that $B$ is a dilation of $A$ simply says that the $(1,1)$ block is given by $A$, that is
$$B=\begin{pmatrix}A&*\\
*& *\end{pmatrix}.$$
Up to scaling down the norm if necessary, it is possible to dilate any operator to a unitary. Indeed, if $\|A\|\leq 1$, we set $K=\H\oplus \H$ and consider the operator
\[
\begin{pmatrix}
A& D_*\\
D&-A^*
\end{pmatrix}, 
\]
where we use the notation  $D=\sqrt{I-A^*A}$ and $D_*=\sqrt{I-AA^*}$. 
Despite its simplicity, this choice of unitary dilation has limited use, since it is unclear just how closely related it is to $A$. For instance, there is no simple formula relating the various powers of $A$ to that of this dilation.

Such a flaw can be remedied, provided that one is willing to enlarge the domain of the dilation. For every positive integer $n$, we let $\H^{(n)}=\H\oplus \ldots \oplus \H$ be the $n$-tuple direct sum of $\H$. We also define an isometric embedding
$
J_n:\H\to \H^{(n)} 
$
as
\[
J_n h=(h,0,\ldots,0).
\]
Identifying $\H$ with $J_n\H$, we see that $\H$ is contained in $\H^{(n)}$.  We define an operator $U_{n,A}:\H^{(n)}\to\H^{(n)}$ as 
\begin{equation}\label{Eq:UnA}
U_{n,A}(h_1,h_2,\ldots,h_n)=(Ah_1+D_*h_n,D h_1 -A^* h_n, h_2, h_3,\ldots,h_{n-1}).
\end{equation}
Equivalently, $U_{n,A}$ has the following block matrix decomposition
\begin{equation}\label{Eq:UnA}
U_{n,A} = 
\begin{pmatrix}
A&0&\dots&0&D_{*}\\
D&0&\dots&0&-A^*\\
0&I&&&0\\
\vdots&&\ddots&&\vdots\\
0&&&I&0
\end{pmatrix}.
\end{equation}
A routine calculation reveals that $U_{n,A}$ is a unitary dilation of $A$. In fact, more is true. Due to the particular structure of the dilation, it follows that $U_{n,A}^k$ is a unitary dilation of $A^k$ for each $0\leq k\leq n-1$ \cite{levy2014dilation}. 

There is an infinite version of this construction that we will also exploit. 
Let $J:\H\to \ell^2(\H)\oplus \ell^2(\H)$ be defined as
\[
Jh=(h,0,0,\ldots) \oplus 0, \quad h\in \H.
\]
This is an isometry, so we may identify $\H$ with $J\H\subset \ell^2(\H)\oplus \ell^2(\H)$.
We first define an operator $V_A:\ell^2(\H)\to \ell^2(\H)$ as
\begin{equation}\label{Eq:VA}
V_A(h_1,h_2,\ldots)=(Ah_1,D h_1, h_2,h_3,\ldots).
\end{equation}
Then, we define another operator $U_A:\ell^2(\H)\oplus \ell^2(\H)\to \ell^2(\H)\oplus \ell^2(\H)$ as
\begin{equation}\label{Eq:UA}
U_A = \begin{pmatrix}
V_A&I-V_AV_A^*\\
0&V_A^*
\end{pmatrix}.
\end{equation}
A standard calculation reveals that $U_A$ is an isometric dilation of $A$. Although the underlying space $\ell^2(\H)$ can be quite large, this dilation enjoys the following important algebraic property:
\[
f(A)=P_{\H} f(U_A)|_\H \quad \text{for every polynomial } f.
\]
This last feature is very appealing. Indeed, the structure of the unitary dilation can be completely understood by means of function theory on the unit circle, using the spectral theorem. In turn, the previous relation shows how information about $U$ can be translated into information about $A$. Deep facts in operator theory can be extracted in this fashion; see \cite{nagy2010harmonic},\cite{paulsen2002completely} and the references therein.

\subsection{Singular values and dilations}

In proving that singular values decay, we will need to compare the growth of singular values of an operator and that of a dilation of it. We thus record the following elementary estimate.

\begin{lemma}\label{L:singvaldil}
Let $\H$ and $\K$ be Hilbert spaces. Let $\H'\subset \K$ be a closed subspace. Let $X\in \B(\H,\H')$ be compact. Let $Y\in \B(\H,\K)$ be a compact operator satisfying $P_{\H'}Y=X$ and let $Z\in \B(\K)$ be an arbitrary compact operator. 
Then, for each pair of integers  $k,\ell$ we have
\[
\frac{s_\ell(X)}{s_{k}(X)} \leq   \left(1+\frac{\|Z-JX\|}{s_k(X)}\right) \frac{(s_\ell(Z)+\|Y-Z\|))}{s_k(Z)}
\]
where $J:\H'\to\K$ is the inclusion map.
\end{lemma}
\begin{proof}
Invoking \cite[Corollary 1.5]{davidsonnest}, the equality $P_{\H'}Y=X$ implies 
$
s_k(X)\leq s_k(Y)
$ for each $k\geq 1$. Moreover, $s_k(JX)=s_k(X)$ for every $k\geq 1$. Next, observe that
\[
s_k(Z) \leq s_k(X)+\|Z-J X\|, \quad s_k(Y)\leq s_k(Z)+\|Z-Y\|
\]
for each $k\geq 1$.
The first of these inequalities implies that
\[
\frac{1}{s_k(X)}\leq \left(1+\frac{\|Z-JX\|}{s_k(X)}\right) \frac{1}{s_k(Z)}, \quad k\geq 1
\]
which, combined with the second inequality, yields
\begin{align*}
\frac{s_\ell(X)}{s_{k}(X)}&\leq \frac{s_\ell(Y)}{s_{k}(X)}\leq \left(1+\frac{\|Z-JX\|}{s_k(X)}\right) \frac{s_\ell(Y)+\|Y-Z\|}{s_k(Z)}.
\end{align*}
\end{proof}

\subsection{Sylvester equations}

Let $\H$ and $\H'$ be Hilbert spaces. Let $A\in \B(\H'),B\in \B(\H)$. The corresponding \emph{Sylvester operator}  $\IS_{A,B}:\B(\H,\H')\to \B(\H,\H')$ is defined as
\[
\IS_{A,B}(X)=AX-XB, \quad X\in \B(\H,\H').
\]
It is clear that $\IS_{A,B}$ is a bounded linear operator on the Banach space $\B(\H,\H')$. The following result will be used repeatedly.

%
%
%

\begin{lemma}\label{surjective}
Assume that $\|A\|\leq 1$ and $\sigma(B)\subset \{z\in \IC:|z|>1\}$. Then, the following statements hold.
\begin{enumerate}[{\rm a)}]
\item The series $\sum_{n=1}^\infty \|B^{-n}\|$ converges to some number $\gamma>0$.
\item $\IS_{A,B}$ is invertible.
\item For every $X\in B(\H,\H')$ we have
$
\|X\|\leq\gamma  \|\IS_{A,B}(X)\|$.
\end{enumerate}
\end{lemma}
\begin{proof}
(i) follows immediately from Lemma \ref{L:spec}, while (ii) follows from \cite[page 2]{bhatia1997and}. Next, fix $X\in B(\H,\H')$. We invoke \cite[Theorem 9.1]{bhatia1997and} to see that
\[
X=\sum_{n=0}^\infty A^n \IS_{A,B}(X) B^{-n-1}
\]
so that
\[
\|X\| \leq \gamma \| \IS_{A,B}(X)\| 
\]
as desired.

\end{proof}
%
%
%

The fundamental property of Sylvester equations underlying Theorem \ref{T:BT} and our main results is the following basic algebraic fact.

\begin{lemma}\label{L:Sylalg}
Let $X\in B(\H,\H')$ and put $C=\IS_{A,B}(X)$. Given any polynomials $p$ and $q$ of degree at most $d$, there are operators $S_1,\ldots,S_d\in B(\H')$ and $T_1,\ldots,T_d\in B(\H)$ such that
\[
p(A)Xq(B)-q(A)Xp(B)=\sum_{j=1}^d S_j C T_j.
\]
\end{lemma}
\begin{proof}
This follows from a routine computation, exactly as in the proof of \cite[Theorem 2.1]{beckermann2019bounds}.
\end{proof}

This has the following useful consequence.

\begin{lemma}\label{L:sylcomp}
Assume that $A$ and $B$ are normal, and that the sequence of Zolotarev numbers $Z_k(\sigma(A),\sigma(B))$ converges to $0$. If $\IS_{A,B}(X)$ is compact, then so is $X$.
\end{lemma}
\begin{proof}
Let $\epsilon>0$ and choose an integer $k$ large enough so that $Z_k(\sigma(A),\sigma(B))<\epsilon$. By definition of the Zolotarev numbers, we may choose also $p,q$ polynomials of degree at most $k$ that do not vanish on $\sigma(A)$ and $\sigma(B)$ respectively, and such that the rational function $r=p/q$ satisfies
\[
\frac{\sup_{z\in \sigma(A)}|r(z)|}{\inf_{z\in \sigma(B)}|r(z)|}<\epsilon.
\]
Now, $r(A)$ and $r(B)$ are both invertible by the spectral mapping theorem, and the spectral theorem yields
\[
\|r(A)\| \|r(B)^{-1}\|=\frac{\sup_{z\in \sigma(A)}|r(z)|}{\inf_{z\in \sigma(B)}|r(z)|}<\epsilon.
\]
Observe next that the operator
\[
K=r(A)Xr(B)^{-1}-X=q(A)^{-1}(p(A)Xq(B)-q(A)Xp(B))p(B)^{-1}
\]
is compact by Lemma \ref{L:Sylalg}. Since $\|X+K\|<\epsilon$, we conclude that $X$ can be approximated in norm by compact operators, and hence it is itself compact.
\end{proof}

\section{Lifting Sylvester equations}\label{dilatesylvester}

This section contains our main dilation results.  Throughout, we have two Hilbert spaces $\H$ and $\H'$, along with operators $A\in \B(\H'),B\in \B(\H)$ and $C\in \B(\H,\H')$. Assume that we are given a solution $X\in B(\H,\H')$ to the Sylvester equation $AX-XB=C$. Next, let $U\in B(\K')$ be a unitary dilation of $A$. We wish to show that the original solution $X$ admits a dilation that solves a ``lifted" Sylvester equation, where $A$  has been replaced by $U$. As mentioned previously, when $C=0$ this is exactly what the so-called  intertwiner lifting theorem accomplishes \cite[Corollary 5.9]{paulsen2002completely}. A different approach is required to handle the general case.

In the following results, we use the specific unitary dilations of $A$ introduced in Section \ref{prelim}. 

%
%

\begin{theorem}\label{T:mainfinite}
Let $\H$ and $\H'$ be Hilbert spaces. Let $A\in \B(\H')$ have norm equal to $1$, and let $B\in \B(\H)$ be such that $\sigma(B)\subset \{z\in \IC:|z|>1\}$. Assume that $C,X\in B(\H,\H')$ satisfy $AX-XB=C$. Then, for every $\epsilon>0$ there is an integer $n\geq 1$ and two operators $Y,Z\in \B(\H,\H'^{(n)})$ such that $P_{J\H'}Y=X$, $\|Z-Y\|<\epsilon$ and $U_{n,A}Z-ZB=J_n C$.
\end{theorem}
\begin{proof}
Throughout the proof, we let $D=\sqrt{I-A^*A}$ and $D_*=\sqrt{I-AA^*}$.
By Lemma \ref{surjective}, the series $\sum_{n=0}^\infty \|B^{-n-1}\|$ converges to some $\gamma>0$. We may thus choose an integer $n\geq 1$ large enough so  that
$
\gamma^2 \|D X\| <\epsilon.
$
Lemma \ref{surjective} also implies that we may find $Y_n\in B(\H,\H')$ satisfying 
\begin{equation}\label{Eq:surjn}
A^*Y_n+Y_nB^{n-1}=DX
\end{equation}
and
\begin{equation}\label{Eq:normYn}
\|Y_n\|\leq \gamma \|D X\|.
\end{equation}
 For each $2\leq k\leq n-1$, define 
 \begin{equation}\label{Eq:Yk}
 Y_k=Y_nB^{n-k}.
 \end{equation}
 Put $Y_1=X$. We may now define $Y:\H\to \H'^{(n)}$ as
\[
Yh=(Y_1 h,Y_2 h,\ldots, Y_n h).
\]
Clearly, we have $P_{J\H'}Y=X$.
Using \eqref{Eq:UnA} along with \eqref{Eq:surjn} and \eqref{Eq:Yk}, we find
\begin{equation}
U_{n,A}Y-YB= \begin{pmatrix}
AX-XB+D_{*}Y_n\\
DX-A^*Y_n-Y_2B\\
Y_2-Y_3B\\
\vdots\\
Y_{n-1}-Y_nB\end{pmatrix}= \begin{pmatrix}
C+D_{*}Y_n\\
0\\
0\\
\vdots\\
0\end{pmatrix}=J_n(C+D_{*}Y_n).
\label{eq:lemmaeq}\end{equation}
Next, invoke Lemma \ref{surjective} once again to find $Z\in B(\H,\H'^{(n)})$ satisfying $U_{n,A}Z-ZB=J_n C.$ It only remains to estimate the size of $\|Y-Z\|$. For this purpose, note that 
\[
U_{n,A}(Y-Z)-(Y-Z)B=J_n D_{*}Y_n.
\]
Applying Corollary \ref{surjective} one more time and invoking \eqref{Eq:normYn}, we find
\begin{align*}
\|Y-Z\|& \leq \gamma  \|J_n D_{*} Y_n\| \leq \gamma^2 \|DX\|< \epsilon
\end{align*}
where the last inequality follows from our choice of $n$.
\end{proof}

In light of the previous result, it is natural to wonder whether the approximation therein can be exact provided we replace $U_{n,A}$ by its counterpart $U_A$ (see \eqref{Eq:UA}). In other words, if we are willing to enlarge $\H'$ to a potentially infinite dimensional Hilbert space, can $\epsilon$ be taken to be $0$?

The next result shows that this is indeed possible.

\begin{theorem}\label{T:maininfinite}
Let $\H$ and $\H'$ be Hilbert spaces. Let $A\in \B(\H')$ have norm equal to $1$, and let $B\in \B(\H)$ such that $\sigma(B)\subset \{z\in \IC:|z|>1\}$. Assume that $C,X\in B(\H,\H')$ satisfy $AX-XB=C$. Then, there is an operator $Y\in \B(\H,\ell^2(\H')\oplus \ell^2(\H'))$ with the following properties.
\begin{enumerate}[{\rm (a)}]
\item $P_{J\H'}Y=X$
\item $U_{A}Y-YB=J C$
\item $\|Y-JX\|<\|X\| \sum_{n=1}^\infty \|B^{-n}\|$
\item $Y$ may be chosen to be compact provided that $X$ is compact.
\end{enumerate}
\end{theorem}
\begin{proof}
Throughout the proof, we let $D=\sqrt{I-A^*A}$ and $D_*=\sqrt{I-AA^*}$.
Put $Z_1=X$ and $Z_2=D X B^{-1}$. For each $n\geq 3$, we recursively define $Z_n=Z_{n-1}B^{-1}$. In other words, $Z_n=D XB^{-(n-1)}$.
For each $m\geq 1$, define $W_m:\H\to \ell^2(\H')$ as
\[
W_m h=(Z_1 h, \ldots, Z_m h,0\ldots ).
\]
Given $n>m\geq 2$, we compute
\[
\|W_n-W_m\|\leq \sum_{j=m+1}^n \|Z_j\| \leq \|X\|  \sum_{j=m+1}^n \|B^{-(j-1)}\|.
\]
By Lemma \ref{surjective}, we infer that the sequence $(W_m)$ is Cauchy, and hence it converges in norm to the operator $Z:\H\to \ell^2(\H')$ such that
\[
Zh=(Z_1 h,Z_2 h,\ldots).
\]
Using \eqref{Eq:VA}, we may now compute
\begin{align*}
V_A Z-ZB&=\begin{pmatrix}AZ_1-Z_1B\\ DZ_1-Z_2 B\\Z_2-Z_3 B\\\vdots \\ Z_n-Z_{n+1}B\\ \vdots \end{pmatrix}=\begin{pmatrix}C\\ 0\\ \vdots  \end{pmatrix}.
\end{align*}
Hence, if we define $Y:\ell^2(\H)\to \ell^2(\H')\oplus \ell^2(\H')$ as
\[
Y v=Zv \oplus 0, \quad v\in \ell^2(\H)
\]
then \eqref{Eq:UA} implies that
\[
U_A Y-ZB=JC.
\]
It is easy to verify that $P_{J\H'}Y=X$.  Furthermore,
\[
\|Y-JX\|\leq \sum_{n=2}^\infty\|Z_n\|\leq \|X\| \sum_{n=1}^\infty \|B^{-n}\|
\]
Finally, when $X$ is compact, then every $W_m$ is compact, so that the limit $Z$ is also compact. This implies that $Y$ is compact as well.
\end{proof}

Under appropriate additional conditions, the arguments used in the proofs of Theorems \ref{T:mainfinite} and \ref{T:maininfinite} can be adapted to also replace the other coefficient $B$ by its unitary dilation. This appears to have limited use for the purpose of establishing the decay of singular values of $X$: we know of no obvious way to guarantee that the unitary dilations  have well separated spectra.  Therefore, we do not pursue this here, and rather refer the interested reader to \cite{klippenstein2022singular} for details and additional dilation results of similar type.

\subsection{Decay of singular values}

We close this section with an application of our operator theoretic results.

\begin{corollary}\label{fastdecay}
Let $\H$ and $\H'$ be Hilbert spaces. Let $A\in \B(\H')$ have norm equal to $1$, and let $B\in \B(\H)$ be self-adjoint with spectrum contained in the interval $[a,b]$ for some $a>1$. Assume that $C,X\in B(\H,\H')$ satisfy $AX-XB=C$, and that $C$ has finite rank equal to $v$. Then, for each positive integer $k$ and $\ell$, we have that
$$\frac{s_{\ell+2vk}(X) }{s_\ell(X)}\leq 4\left[1+\frac{\norm{X}}{s_\ell(X)(a-1)} \right]\left[\exp\left(\frac{\pi^2}{2\ln (16\gamma)}\right)\right]^{-2k}$$
for
$$ \gamma = \left[\frac{(a+1)(1-b)}{(1-a)(b+1)}\right]^2.$$
\end{corollary}
\begin{proof}
By Theorem \ref{T:maininfinite}, there is a Hilbert space $\K$ containing $\H'$, an operator $Y\in \B(\H,\K)$ satisfying $P_{\H'}Y=X$, a rank $v$ operator $D\in \B(\H,\K)$ and a unitary $U\in \B(\K)$  such that
$$UY-YB=D$$
and
\[
\|Y-JX\|\leq \|X\| \sum_{n=1}^\infty \|B^{-n}\|=\|X\| \sum_{n=1}^\infty a^{-n}=\frac{\|X\| }{a-1}.
\] 
Since the spectrum of $U$ is contained in the unit circle $\IT$, using Theorem \ref{T:BT} for each positive integer $k$ and $\ell$, we obtain
$$s_{\ell+2vk}(Y)\leq Z_{2k}\left(\IT, [a,b]\right)s_{\ell}(Y).$$
Next, Corollary \ref{lineandcircleboundcorollary} gives us an upper bound for the Zolotatev numbers. In particular, from Lemma \ref{L:sylcomp} we infer that $Y$ must be compact, and thus so is $X=P_{\H'}Y$.
It thus only remains to apply Lemma \ref{L:singvaldil} with $Y=Z$ to obtain the desired estimate.
\end{proof}

\subsection{An example}
We saw in the introduction that, generally speaking, no information about the behaviour of the singular values of $X$ can be extracted from the fact that $AX-XB$ has small rank. On the other hand, a decaying condition is obtained in Theorem \ref{T:BT}, provided that $A$ and $B$ are normal with well-separated spectra. The normality condition alone is not sufficient: take for instance $A=C$ to be a rank one projection, $B=0$ and $X=I$. The aim of this subsection is to show that the spectral condition alone is not sufficient.

Fix $n\geq 2$. For $0<b<1$, we define $T_b\in \IC^{n\times n}$ 
to be the usual upper-triangular Jordan block with eigenvalue $b$.  Put $B_b=T_b^{-1}$. 

\begin{example}\label{counterexample}
Let $X=I$, the identity matrix of size $n\geq 2$. Standard estimates can be used to show that there is $0<b<1$ small enough so that 
\[
\sum_{k=1}^{\infty} s_2(B_b)^k\norm{T_b^k} \leq n,
\]
see \cite[Example 5.18]{klippenstein2022singular} for details. 
In particular, it follows from Lemma \ref{L:spec} that the spectrum of $B_b/s_2(B_b)$ is contained in $\{z\in \IC:|z|>1\}$.
Let $R$ be a rank $1$ matrix such that $\|B_b-R\|=s_2(B_b)$. Taking $A=\frac{1}{s_2(B_b)}(B_b-R)$ and $C=-\frac{1}{s_2(B_b)}R$, we thus have
\[
AX-X\left( \frac{1}{s_2(B_b)}B_b\right)=C.
\]
Invoking Theorem \ref{T:maininfinite}, there is a Hilbert space $\K$, an isometric embedding $J:\IC^n\to \K$, an operator $Y\in \B(\IC^n,\K)$ satisfying $P_{J\IC^n}Y=X$, a rank $1$ operator $D\in \B(\IC^n,\K)$ and a unitary $U\in \B(\K)$  such that
$$UY-Y\left( \frac{1}{s_2(B_b)}B_b\right)=D$$
and
\[
\|Y-JX\|<\sum_{k=1}^\infty s_2(B_b)^k \|T_b^{-k}\|.
\]
In view of our choice of $b$, we infer that 
$
\|Y-JX\|<n.
$
We see that $Y$ satisfies a Sylvester equation whose right-hand side has rank $1$, and for which the spectra of the coefficients are far from each other. Nevertheless, the singular values of $Y$ decay rather slowly.  Indeed, Lemma \ref{L:singvaldil} implies that
\[
\frac{s_k(Y)}{s_1(Y)}\geq \frac{1}{1+n}
\]
for each $k\geq 1$.
\end{example}

%
%
%
%

\section{Solving Sylvester equations}\label{applications}

There are many known methods for solving Sylvester equations involving  finite dimensional operators. The most well-known algorithm is the Bartels--Stewart method, which consists of taking the Schur decomposition of both coefficients $A$ and $B$ \cite{bartels1972solution}. When $A$ and $B$ have disjoint spectra, the solution is necessarily unique, and  this method always produces an exact solution. However, there is one disadvantage.  Even if $A$ and $B$ are sparse, the unitary factors in their Schur decomposition will almost always be dense. This in turn causes the computational complexity to be $\mathcal{O}(m^3+n^3)$, where $m$ and $n$ are the sizes of $A$ and $B$ respectively.

\subsection{Alternating Direction Implicit Method}\label{adi}

One method which avoids the problems with the Bartels--Stewart method is an iterative algorithm known as the alternating direction implicit (ADI) method, which was first introduced in 1955 in \cite{peaceman1955numerical}. For $k$ iterations, the first step is to choose complex shifts $\{\alpha_i\}_{i=1}^k$ and $\{\beta_j\}_{j=1}^k$ such that $A-\beta_j$ and $B-\alpha_j$ are invertible for each $j$.  Next, for an initial guess $X_0$, perform the following steps.
\begin{enumerate}
\item Solve for $X_{j-\frac{1}{2}}$ in
$$(A-\beta_{j})X_{j-\frac{1}{2}} = X_{j-1}(B-\beta_{j})+C.$$
\item Solve for $X_{j}$ in
$$X_{j}(B-\alpha_{j})=(A-\alpha_{j})X_{j-\frac{1}{2}}-C$$
\end{enumerate}

Next, we need to determine how many iterations we need if we want $\norm{X-X_k}$ to be small enough; see \cite{benner2014computing} for several properties of this method after $k$ iterations. The property we are most interested in here is
$$X-X_k = \prod_{j=1}^k\frac{A-\alpha_j}{A-\beta_j}(X-X_0)\prod_{j=1}^k\frac{B-\beta_j}{B-\alpha_j}$$
where our abuse of notation is that division by a matrix means multiplication by its inverse.  Therefore,
$$\norm{X-X_k} \leq \norm{r_k(A)}\norm{r_k(B)^{-1}}\norm{X-X_0} , \qquad r_k(z) = \prod_{j=1}^k \frac{z-\alpha_j}{z-\beta_j}. $$

It immediately follows that if $A$ and $B$ are both normal, and we make an optimal choice of shifts,
\begin{equation}\norm{X-X_k} \leq Z_k\left(\sigma(A),\sigma(B)\right)\norm{X-X_0}.\label{eq:adierror}\end{equation}

For time complexity, clearly it is $\O\left(N\left(m^3+n^3\right)\right)$ in general where $N$ is the total number of iterations. However, if both $A$ and $B$ support fast shifted linear solves, then we can use their structure to solve each iteration in $\O\left(m^2+n^2\right)$.  For example,  if $A$ and $B$ are tridiagonal matrices, we can use Thomas' algorithm \cite{ford2014numerical}.  This implies that the total time is $\O\left(N\left(m^2+n^2\right)\right)$, and thus beats Bartels-Stewart provided not too many iterations are required. Furthermore, if the number of iterations is independent of $m$ and $n$, we can say the time complexity is $\O\left(m^2+n^2\right)$.

On the other hand, a direct computation shows that
\begin{equation}X_{k} = \frac{A-\alpha_k}{A-\beta_k}X_{k-1}\frac{B-\beta_k}{B-\alpha_k} + (\beta_k-\alpha_k)(A-\beta_k)^{-1}C(B-\alpha_k)^{-1}\label{eq:alternateadi}\end{equation}
which implies that if $\rank(C)=v$,  then $\rank(X_k) \leq \rank(X_{k-1}) +v$,  and thus
$$\rank(X_k) \leq kv + \rank(X_0).$$
By combining equations \eqref{eq:adierror} with \eqref{eq:alternateadi}, we get the following interesting consequence.

\subsection{Factored Alternating Direction Implicit Method}

As the name suggests, the factored ADI (fADI) method is similar to the ADI method.  However,  the right-hand side must be factored, and in turn,  the approximate solution will also be factored.  More precisely, if the right-hand side is factored as $C=FG^*$ for $F\in\IC^{n\times r}$ and $G\in\IC^{m\times r}$, then $k$ iterations of the fADI results in $Y_k\in\IC^{n\times kr}$ and $Z_k\in\IC^{m\times kr}$ where $Y_kZ_k^*$ approximately solves $AX-XB=FG^*$.  One must be cautious when using fADI as there are different versions, with different error analyses, in the literature; see for instance \cite{benner2014computing,  benner2009adi}.  The version we use here can be found in \cite{benner2005}.

The first step of fADI consists, once again, of choosing shifts $\{\alpha_i\}_{i=1}^k$ and $\{\beta_j\}_{j=1}^k$ so that $A-\beta_j$ and $B-\alpha_j$ are both invertible for each $j$. Then, with initial guesses $Z_0, Y_0$, perform the following:

\begin{enumerate}
\item Solve for $Y_{j}$ in
$$(A-\beta_{j})Y_{j}=\begin{pmatrix} F & (A-\alpha_{j})Y_{j-1}\end{pmatrix}.$$
\item Solve for $Z_{j}$ in
$$(B-\alpha_j)^*Z_j = \begin{pmatrix} (\beta_j-\alpha_j)^*G & (B-\beta_j)^*Z_{j-1}\end{pmatrix}.$$
\end{enumerate}
After each iteration, define $X_{j} = Y_{j}Z_{j}^*$.

At first glance, the relation between ADI and fADI is not clear. However, one can show that \eqs{alternateadi} can also be derived from the fADI method. This implies the error analysis of fADI is identical to that of ADI.

In general, the time complexity of the fADI method is identical to that of the ADI method. However, if both $A$ and $B$ support fast shifted linear solves, then $N$ iterations require $\mathcal{O}(Nr(m+n))$ time.  Furthermore, if the number of iterations is independent of $m$ and $n$, and $N\ll m,n$, then we can solve each iteration, and thus fADI all together, in $\mathcal{O}(m+n)$ time.

\subsection{Application: convergence of the methods in the non-normal case}

We aim to show that the ADI and fADI method can converge quickly without requiring both coefficients to be normal. As mentioned above, the time complexity of each iteration of the ADI and fADI method is dependent on the dimensions of $A$ and $B$.  Hence,  if we use Theorems \ref{T:mainfinite} and \ref{T:maininfinite}, we must avoid the need to run ADI and fADI on the resulting lifted Sylvester equations, as each iteration could take a very long time or even infinite time. We show here how we can avoid this problem. We start with a technical observation.
%

\begin{lemma}\label{adidilation} 
Let $\H$ and $\H'$ be Hilbert spaces. Let $A\in \B(\H'),B\in \B(\H)$ and $X,C\in \B(\H,\H')$ satisfy $AX-XB=C$. Assume that there is a Hilbert space $\K$ containing $\H'$, along operators $U\in \B(\K),Y\in \B(\H,\K), D\in \B(\H,\K)$ such that $UY-YB=D$ and
\[
P_{\H'}Y=X, \quad P_{\H'}D=C, \quad P_{\H'}U|_{\H'}=A.
\]
Run ADI on both $AX-XB=C$ and $UY-YB=D$ with shifts $\{\alpha_i\}_{i=1}^k$ and $\{\beta_j\}_{j=1}^k$ such that $\{\alpha_i\}_{i=1}^k\cap\sigma(B)=\{\beta_j\}_{j=1}^k\cap\sigma(A)=\{\beta_j\}_{j=1}^k\cap\sigma(U)=\emptyset$. Assume $U$ and $Y_0$ are such that
\begin{equation}P(U-\gamma)Y_j = (A-\gamma)PY_j\label{eq:adirequirement}\end{equation}
for each $\gamma= \alpha_j, \beta_j$ and $1\leq j\leq k$. If $X_0=P_{\H'}Y_0$, then $X_j=P_{\H'}Y_j$ for all $j\geq 0$.
\end{lemma} 
\begin{proof} 
Proceed with induction, noting that the base case holds by assumption. Next, assume that $X_j = P_{\H'}Y_j$. Beginning with applying ADI to $AX-XB=C$, observe:
\begin{equation}\begin{aligned}
(A-\beta_{j+1})X_{j+\frac{1}{2}} &= X_j(B-\beta_{j+1})+C\\
&= P_{\H'}Y_j(B-\beta_{j+1})+P_{\H'}D\\
&=P_{\H'}\left[Y_j(B-\beta_{j+1})+D\right]\\
&=P_{\H'}(U-\beta_{j+1})Y_{j+\frac{1}{2}}\\
&=(A-\beta_{j+1})P_{\H'}Y_{j+\frac{1}{2}}.
\end{aligned}\end{equation}
Since $\beta_{j+1}$ is chosen so that $A-\beta_{j+1}$ is invertible, $X_{j+\frac{1}{2}} =P_{\H'}Y_{j+\frac{1}{2}}$. Showing $X_{j+1} = P_{\H'}Y_{j+1}$ follows from a similar calculation:
\begin{equation}\begin{aligned}
X_{j+1}(B-\alpha_{j+1}) &= (A-\alpha_{j+1})X_{j+\frac{1}{2}} - C\\
&= (A-\alpha_{j+1})P_{\H'}Y_{j+\frac{1}{2}} - P_{\H'}D\\
&=P_{\H'}(U-\alpha_{j+1})Y_{j+\frac{1}{2}}-P_{\H'}D\\
&=P_{\H'}\left[(U-\alpha_{j+1})Y_{j+\frac{1}{2}}-D\right]\\
&=P_{\H'}PY_{j+1}(B-\alpha_{j+1}).
\end{aligned}\end{equation}
As $\alpha_{j+1}$ was chosen so that $B-\alpha_{j+1}$ is invertible, $X_{j+1} = PY_{j+1}$.
\end{proof}

We now arrive at our convergence result.

\begin{theorem}\label{adidilationexample}
Let $A\in \IC^{m\times m}, B\in \IC^{n\times n}$ and $C,X\in \IC^{m\times n}$. Assume that $\|A\|$ has norm equal to $1$, that $B$ is Hermitian with $\sigma(B))\subset[a,b]$ where $a>1$. Assume also that $AX-XB=C$.
 Run $k$ iterations of ADI with an initial guess $X_0=0$, along with shifts $\alpha_i=\frac{2}{a+b}$ and $\beta_j=\frac{a+b}{2}$. Then,
$$\norm{X-X_k}\leq \left(1+\frac{\sqrt{2}}{a-1}\right)\left(\frac{\frac{b}{a}-1}{a+b-\frac{2}{a}}\right)^k\norm X.$$
\end{theorem}
\begin{proof}
First,  observe that if $X$ is zero,  then the result is trivial,  and thus we assume that $X$ is nonzero.  Next, we may apply Theorem \ref{T:maininfinite} and get a lifted equation $UY-YB=D$ that satisfies
\[
\|Y-JX\|\leq \|X\|\sum_{n=1}^\infty \|B^{-n}\|=\|X\|\sum_{n=1}^\infty a^{-n}=\frac{\|X\|}{a-1}.
\]
Run ADI on $UY-YB=D$ with the same shifts and initial guess $Y_0=0$. A straightforward yet tedious calculation shows that \eqs{adirequirement} holds. By Lemma \ref{adidilation}, there is an orthogonal projection $P$ such that  $X_j = PY_j$ for $j\geq 1$. Furthermore, we have that
\[
\|Y\|\leq \|JX\|+\|Y-JX\|=\|X\|+\|Y-JX\|
\]
whence
\[
\frac{1}{\|X\|}\leq \frac{1}{\|Y\|} \left(1+\frac{\norm{Y-JX}}{\norm{X}}\right).
\]
Consequently, by equations \eqref{eq:bound2} and \eqref{eq:adierror} we find
$$\begin{aligned}
\frac{\norm{X-X_k}}{\norm{X}}&\leq \left(1+\frac{\norm{Y-JX}}{\norm{X}}\right)\frac{\norm{PY-PY_k}}{\norm{Y}} \\
&\leq \left(1+\frac{\norm{Y-X'}}{\norm{X}}\right)\frac{\norm{Y_k-Y}}{\norm{Y}} \\
&\leq \left(1+\frac{1}{a-1}\right)\left(\frac{\frac{b}{a}-1}{a+b-\frac{2}{a}}\right)^k.
\end{aligned}$$
\end{proof}

Recall from \eqs{alternateadi} that $X_k$ obtained from both ADI and fADI are the same, and Corollary \ref{adidilationexample} holds if we replace ADI with fADI.  Moreover, one can use Corollary \ref{lineandcircleboundcorollary} instead of \eqs{bound2} to bound the Zolotarev numbers.  In this case,  the shifts would be the optimal shifts over two intervals,  \cite{fortunato2020fast}, and could then be related back to the unit circle and interval according the transformations used in Theorem \ref{lineandcirclebound}.

\section{Numerical Experiments}\label{numericalresult}
In this section, we use Theorem \ref{adidilationexample} to numerically solve a partial integro-differential equation. One class of these equations involve what is known as a Volterra integral operator, see \cite{corduneanu2000volterra, volterra1959theory}. These equations arise in fields such as fluid mechanics, biophysics and naval architecture. Here, we solve a simplified partial integro-differential equation which results in a Sylvester equation with one nonnormal coefficient.

Consider
\begin{equation}\int_0^x a(t)u(t,y)dt+u_{yy}(x,y)=f(x,y) ,\qquad u(x,0)=u(x,1)=0\label{eq:differential}\end{equation}
where $x,y\in [0,1]$, $f(x,y)$ is continuous,  and $a(t)$ is twice differentiable.  Further, denote
$$a_{max}=\sup_{t\in  [0,1]} |a(t)|$$
and assume
$$a_{max} \leq 4.$$
As we shall see,  this ensures our Sylvester equation will have a unique solution.

We begin with discretizing $x$ and $y$ into $n$ equispaced points from $1/(n+1)$ to $n/(n+1)$. See, for example, \cite{faires2003numerical}, for methods to numerically compute derivatives and integrals along with their error bounds. We approximate the integral with the trapezoidal rule, and the second partial derivative with the centered difference to obtain
$$\int_0^{x_i} a(t)u(t,y)dt \approx\frac{a_1u_{1,j}+a_iu_{i,j}}{2(n+1)} +\sum_{k=2}^{i-1} \frac{a_ku_{k,j}}{n+1},\qquad \frac{\d^2}{\d y^2}u(x_i,y)\Big\rvert_{y_j} \approx (n+1)^2\left[u_{i,j-1}-2u_{i,j}+u_{i,j+1}\right].$$
Thus, the corresponding Sylvester equation is
$$\underbrace{\frac{S}{(n+1)}}_{A}U+U\underbrace{(n+1)^2T}_{-B}=F$$
where
$$S=\begin{pmatrix}\frac{a_1}2\\ \frac{a_1}2&\frac{a_2}2\\\vdots&a_2&\ddots\\\vdots&a_2&a_3&\ddots\\\frac{a_1}2&a_2&\dots&a_{n-1}&\frac{a_n}2\end{pmatrix},\qquad T=\begin{pmatrix}-2&1\\1&\ddots&\ddots\\&\ddots&\ddots&1\\&&1&-2\end{pmatrix},  \qquad a_k = a\left(\frac{k}{n+1}\right),\qquad F_{i,j} = F\left(\frac{i}{n+1},\frac j{n+1}\right).$$

The matrix $B$ is a well-known Toeplitz symmetric tridiagonal matrix with eigenvalues given by
$$\lambda_k(B) =4(n+1)^2\sin^2\left(\frac{k\pi}{2(n+1)}\right).$$
From here,  an elementary calculation using Jordan's inequality shows that if $A$ and $B$ had a common eigenvalue,  we would require that $a_j \geq 16$ for some $j$.  However,  this contradicts the assumption that $a_{max}\leq 4$, and thus we do indeed get a unique solution.

Using Jordan's inequality, we have
$$\norm{B^{-1}}=\frac1{4(n+1)^2}\frac1{\sin^2\left(\frac{\pi}{2(n+1)}\right)}\leq \frac1{4(n+1)^2}(n+1)^2=\frac14.$$
On the other hand,  notice that $\norm{A}$ is at most $a_{max}\norm{S'}/(n+1)$ where $S'$ is the $n\times n$ lower triangular matrix where each entry on and below the main diagonal is one. To determine $\norm{S'}$, we compute the smallest singular value of $S'^{-1}$. To do this, we start by noticing that
$$S'^{-1^*}S'^{-1}=\begin{pmatrix} 1&-1\\-1&2&\ddots\\&\ddots&\ddots&-1\\&&-1&2\end{pmatrix}=2I -2\underbrace{\begin{pmatrix} \frac12&\frac12\\\frac12&0&\ddots\\&\ddots&\ddots&\frac12\\&&\frac12&0\end{pmatrix}}_J.$$
Since $J$ is the Jacobi matrix associated with the Chebyshev polynomials of the fourth kind, we can see that the eigenvalues of $J$ are the roots of the $n^{\rm th}$ degree Chebyshev polynomial of the fourth kind, $W_n$. Further, see equation (18.5.4) of \cite{NIST:DLMF} for the explicit form of $W_n$ which is
$$W_n(x)=\frac{\sin\left(\left[n+\frac12\right]\cos^{-1}x\right)}{\sin\left(\frac12\cos^{-1}x\right)}.$$
Thus, the roots of $W_n$, and the eigenvalues of $J$, have the form
$$\lambda_k(J)=-\cos\left(\frac{k\pi}{n+\frac12}\right).$$
Hence,
$$\lambda_k\left(S'^{-1^*}S'^{-1}\right)=2+2\cos\left(\frac{k\pi}{n+\frac12}\right)=4\cos^2\left(\frac{k\pi}{2n+1}\right)$$
which implies
$$\norm{S'} = \frac1{2\cos\left(\frac{n}{2n+1}\pi\right)}=\frac1{2\sin\left(\frac\pi2-\frac{n}{2n+1}\pi\right)}=\frac1{2\sin\left(\frac\pi{4n+2}\right)}\leq n+\frac12$$
where we used Jordan's inequality once more. Thus,
\begin{equation}\norm{A}\leq \frac{a_{max}}{n+1}\norm{S'}=\frac{a_{max}}{2(n+1)\sin\left(\frac\pi{4n+2}\right)} \leq \frac{n+\frac12}{n+1}a_{max}\leq a_{max}\label{eq:bound1}\end{equation}
and,
$$\norm{A}\norm{B^{-1}}\leq \frac{a_{max}}{8(n+1)^3\sin\left(\frac\pi{4n+2}\right)\sin^2\left(\frac{\pi}{2(n+1)}\right)}\leq \frac14a_{max}.$$
We briefly mention that the second inequality is very weak. In fact, the middle term is approximately $2a_{max}/\pi^3$.

Next, assuming $a_{max}>0$, we take $AU-UB=F$ and divide by $a_{max}$, which gives us $A'U-UB'=F'$ where $\norm{A'}\leq 1$ and $\sigma(B')\subset [a,b]$ for
\begin{equation}\frac4{a_{max}}\leq a=\frac{4(n+1)^2\sin^2\left(\frac{\pi}{2(n+1)}\right)}{a_{max}} \leq \frac{\pi^2}{a_{max}}\label{eq:a}\end{equation}
and
\begin{equation}\frac{4n^2}{a_{max}}\leq  b=\frac{4(n+1)^2\sin^2\left(\frac{n\pi}{2(n+1)}\right)}{a_{max}} \leq \frac{n^2\pi^2}{a_{max}}.\label{eq:b}\end{equation}

Therefore, by Corollary \ref{adidilationexample}, the relative error after the $k^{\rm}$ iteration of ADI is
\begin{equation}\frac{\norm{U_k-U}}{\norm U} \leq \left(1+\frac{\sqrt  2}{a-1}\right)\left(\frac{\frac ba-1}{a+b-\frac2a}\right)^k\leq \left(1+\frac{\sqrt2}{4-a_{max}}a_{max}\right)\left(\frac{n^2\pi^2-4}{16+16n^2-2a_{max}^2}a_{max}\right)^k.\label{eq:relativeerror}\end{equation}
If we assume further that $a_{max}\leq \sqrt{8+32/\pi^2}$, then we can simplify this to
\begin{equation}\frac{\norm{U-U_k}}{\norm U} \leq  \left(1+\frac{\sqrt2}{4-a_{max}}a_{max}\right)\left(\frac{\pi^2}{16}a_{max}\right)^k.\label{eq:simplerbound}\end{equation}
If we want to ensure a relative error of $\epsilon$, we require $N$ iterations for
$$N=\Bigg\lceil \frac{\ln\left(\frac{\epsilon}{1+\frac{\sqrt2}{a-1}}\right)}{\ln\left(\frac{\frac ba-1}{a+b-\frac2a}\right)}\Bigg\rceil.$$
Alternatively, using the simpler bound in \eqs{simplerbound} we have
$$N=\Bigg\lceil \frac{\ln\left(\frac{\epsilon}{1+\frac{\sqrt2}{4-a_{max}}a_{max}}\right)}{\ln\left(\frac{\pi^2}{16}a_{max}\right)}\Bigg\rceil.$$
This proves that the number of iterations required is independent of $n$. Further,  note that the off diagonal component of $A$ has rank-one. This in turns allows the equation $Ax=b$ to be solved in linear time. Additionally, we can use Thomas' algorithm, see \cite{ford2014numerical}, to solve $B^*y=c$ in linear time as well. Therefore, the time complexity for fADI is $\mathcal{O}(n)$ if $f(x,y)$ is a low rank function.

Additionally, we can use Corollary \ref{fastdecay} to bound the decay of the singular values of $X$. If $U$ has unit norm, Corollary \ref{fastdecay} states that
\begin{equation}s_{1+2vk}(U) \leq 4\left[1+\frac{\sqrt{2}}{a-1} \right]\left[\exp\left(\frac{\pi^2}{2\ln (16\gamma)}\right)\right]^{-2k},\qquad  \gamma = \left[\frac{(a+1)(1-b)}{(1-a)(b+1)}\right]^2.\label{eq:decay}\end{equation}

Now we solve \eqs{differential} for a given $a(t)$ and $f(x,y)$.

\begin{example}
Consider 
$$\int_0^x a(t)u(t,y)dt +u_{yy}= f(x,y) ,\qquad u(x,0)=u(x,1)=0$$
for 
$$a(t)=t, \qquad f(x,y)=6x^4y(1-2y)+\frac16 x^6y^3(1-y).$$
This has solution
$$u(x,y)=x^4y^3(1-y).$$
To solve it numerically with the method above, first notice that $a_{max}=1$ and thus
$$\norm A\norm{B^{-1}} \leq \frac14.$$
Further, we have
$$\frac{\norm{U_k-U}}{\norm U} \leq  \left(1+\frac{\sqrt2}{3}\right)\left(\frac{\pi^2}{16}\right)^k.$$
Therefore, if we want to ensure a relative error of $\epsilon$, we require an amount of iterations given by
\begin{equation}N=\Bigg\lceil \frac{\ln\left(\frac{\epsilon}{1+\frac{\sqrt2}{3}}\right)}{\ln\left(\frac{\pi^2}{16}\right)}\Bigg\rceil\label{eq:required}\end{equation}
We illustrate the speed and accuracy of the fADI method in the following figure and table.

\begin{table}[H]
\centering
\begin{tabular}{ |p{2.5cm}||p{2.5cm}|p{2.5cm}|p{2.5cm}|p{2.5cm}|p{2.5cm}|}
 \hline
$n$& $k=1$ & $k=2$ & $k=3$&$k=4$&$k=5$\\
\hline

&&&&\\[-1em]
$10^1$&$0.085557058$&$0.008182140$&$0.000662944$&$0.000046905$&$0.000002980$\\
\hline
&&&&&\\[-1em]
$10^2$&$0.098737674$&$0.010548343$&$0.000938914$&$0.000071019$&$0.000004667$\\
\hline
&&&&&\\[-1em]
$10^3$&$0.099677737$&$0.010746601$&$0.000965172$&$0.000073639$&$0.000004879$\\
 \hline
\end{tabular}
\caption{We illustrate the ratio of our numerical data for the relative error with the first bound in \eqs{relativeerror} for the exact values of  $a$ and $b$ given in \eqs{a} and \eqs{b}.}
\end{table}

\begin{figure}[H]
\centering
\includegraphics[scale=0.6]{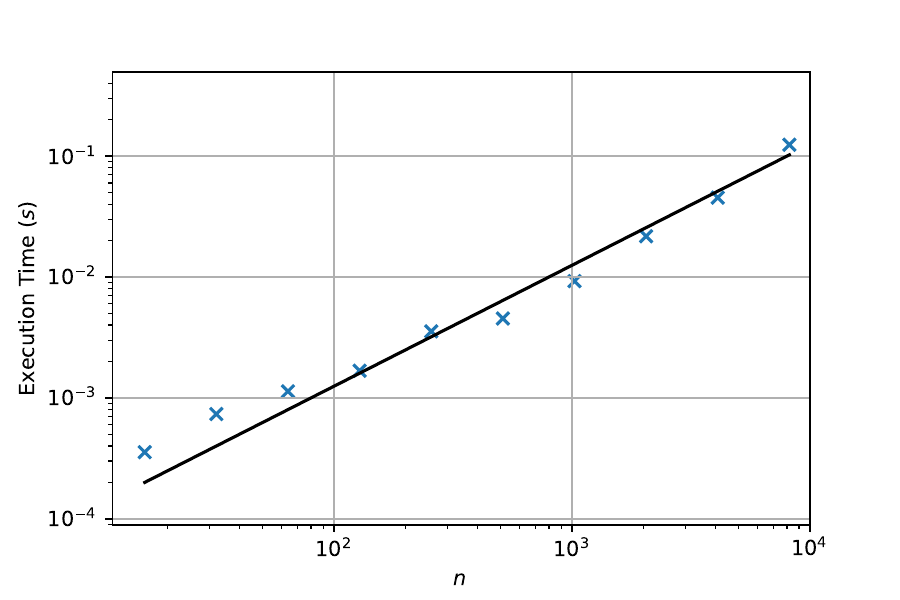}
\caption{This compares the execution time required to ensure a relative error of at most $10^{-7}$. The dots represent data for different values of $n$, and the line represents $\mathcal{O}(n)$. The number of iterations was chosen by \eqs{required}.}
\end{figure}

Moreover, we can use \eqs{decay} to bound the decay of the singular values of $U$. However, in this example, we have that $s_k(U)$ is effectively zero for any $k\geq 3$ regardless of the size of $n$.

\end{example}

\section{Conclusion}
In this paper, we began with the result by Beckermann and Townsend which says that if we have a Sylvester equation with normal coefficients and have well-separated spectra, and the rank of the right-hand side is low, then the solution has a low rank approximation. We explored the possibility of relaxing the normality condition, and thanks to Corollary \ref{fastdecay}, we were able to trade the normality of one of the coefficients for a norm condition on both coefficients. Additionally,  Theorem \ref{adidilationexample} illustrates how Sylvester equations can be solved quickly without requiring normality of both coefficients. Finally, we give two open questions.
\begin{enumerate}
\item Theorem \ref{lineandcirclebound} gives us an upper bound on the Zolotarev numbers over the unit circle and an interval. However, Figure \ref{fig:comparisons} shows that this bound is not ideal. Thanks to a M\"{o}bius transform, this problem is equivalent to finding the Zolotarev numbers over the real line and an imaginary interval. It should be noted that this problem is also motivated by the fact that if one has a Sylvester equation $AX-XB=C$ where one of $A$ and $B$ is symmetric and the other is skew-symmetric, then one needs the Zolotarev numbers over a real interval and an imaginary interval.
\item Given the Sylvester equation $AX-XB=C$ where $\norm{A}=1$, $B$ is normal and the spectrum of $B$ is contained inside the unit circle, then can we conclude anything about the decay of the singular values of $X$? This problem is well motivated for two reasons. First, numerical experiments show that this is indeed often the case. Second, this would effectively complete the case where $AX-XB=C$ and $B$ is normal. To understand why, first by unitarily diagonalizing $B$, we can assume
$$B=\begin{pmatrix}B_1&\\&B_2\end{pmatrix}$$
where the spectrum of $B_1$ is inside the unit circle, and $B_2$ is outside the unit circle, assuming the spectrum of $B$ does not intersect the unit circle. Next, we can write $AX-XB=C$ as
$$A\begin{pmatrix}X_1&X_2\end{pmatrix}-\begin{pmatrix}X_1&X_2\end{pmatrix}\begin{pmatrix}B_1&\\&B_2\end{pmatrix}=\begin{pmatrix}C_1&C_2\end{pmatrix}.$$
From here, we obtain two Sylvester equations,
$$AX_1-X_1B_1=C_1 ,\qquad AX_2-X_2B_2=C_2.$$
By Corollary \ref{fastdecay}, we can conclude $X_2$ has fast decay of its singular values. Additionally, if one can find conditions which give a positive answer to this question, then we would also have fast decay of the singular values of $X_1$. Finally, we can use the Eckart-Young theorem to show that
$$\frac{s_{2k+1}(X)}{s_\ell(X)} \leq \frac{s_{k+1}(X_1)}{s_{\ell}(X_1)}+\frac{s_{k+1}(X_2)}{s_\ell(X_2)}$$
which would prove that $X$ also has fast decay of its singular values.

We also point out that the answer to this question cannot always be positive without imposing further conditions. This follows from taking the adjoint of the Sylvester equation in example \ref{counterexample}.
\end{enumerate}

%
%
%

\bibliographystyle{IEEEtran}
\bibliography{references}
\end{document}